\documentclass[11pt,leqno]{amsart}
\usepackage{amsmath,amssymb,amsthm}
\usepackage{hyperref}

\newtheorem{theorem}{Theorem}[section]
\newtheorem{lemma}[theorem]{Lemma}

\newtheorem{proposition}[theorem]{Proposition}
\newtheorem{corollary}[theorem]{Corollary}
\theoremstyle{definition}
\newtheorem{definition}[theorem]{Definition}

\newtheorem{problem}[theorem]{Problem}

\theoremstyle{remark}
\newtheorem{remark}[theorem]{Remark}
\numberwithin{equation}{section}

\def\fnote#1{\footnote}

\def\ignora#1{}
\def\n3#1{\left\vert  \! \left\vert \! \left\vert \, #1 \, \right\vert \!
  \right\vert \! \right\vert }

\renewcommand{\geq}{\geqslant}
\renewcommand{\leq}{\leqslant}

\newcommand{\Free}{{\mathcal F}}

\newcommand{\Lip}{{\mathrm{Lip}}_0}

\newcommand{\supp}{\operatorname{supp}}
\newcommand{\spa}{\operatorname{span}}

\newcommand{\pten}{\ensuremath{\widehat{\otimes}_\pi}}

\begin{document}



\title{ Octahedral norms in duals and biduals of Lipschitz-free spaces }

\author {Johann Langemets}
\address{Institute of Mathematics and Statistics, University of Tartu, Narva mnt 18, 51009 Tartu, Estonia}
\email{johann.langemets@ut.ee}
\thanks{The work of Johann Langemets was supported by the Estonian Research Council grant (PSG487) and by the European Regional Development Fund and the programme Mobilitas Pluss (MOBTP138).}
\urladdr{\url{https://johannlangemets.wordpress.com/}}

\author{Abraham Rueda Zoca }\thanks{The research of Abraham Rueda Zoca  was supported by Vicerrectorado de Investigaci\'on y Transferencia de la Universidad de Granada in the program ``Contratos puente”, by MICINN (Spain) Grant PGC2018-093794-B-I00 (MCIU, AEI, FEDER, UE), by Junta de Andaluc\'ia Grant A-FQM-484-UGR18 and by Junta de Andaluc\'ia Grant FQM-0185.}
\address{Universidad de Granada, Facultad de Ciencias.
Departamento de An\'{a}lisis Matem\'{a}tico, 18071-Granada
(Spain)} \email{abrahamrueda@ugr.es}
\urladdr{\url{https://arzenglish.wordpress.com}}

\maketitle

\begin{abstract}
We continue with the study of octahedral norms in the context of spaces of Lipschitz functions and in their duals. First, we prove that the norm of $\mathcal F(M)^{**}$ is octahedral as soon as $M$ is unbounded or is not uniformly discrete. Further, we prove that a concrete sequence of uniformly discrete and bounded metric spaces $(K_m)$ satisfies that the norm of $\mathcal F(K_m)^{**}$ is octahedral for every $m$. Finally, we prove that if $X$ is an arbitrary Banach space and the norm of $\Lip(M)$ is octahedral, then the norm of $\Lip(M,X^*)$ is octahedral. These results solve several open problems from the literature.
\end{abstract}

\section{Introduction}

The Lipschitz-free space $\mathcal F(M)$ of a metric space $M$ is a Banach space with the property that every Lipschitz function admits a canonical linear extension defined on $\mathcal F(M)$. This linearisation property makes Lipschitz-free spaces a useful tool to study Lipschitz maps between metric spaces and, because of this reason, a big effort to understand its Banach space structure has been made in the last 20 years (see \cite{godefroysurvey, kp, weaver} and references therein). 
The study of geometric properties of Lipschitz-free spaces has experimented a very intense and recent activity (to mention few of it, see \cite{blrlipschitz, pr} for results on octahedrality, \cite{am,gprdaugavet,ikw} for results about Daugavet property, \cite{ag,ap,gprdaugavet,pp,weaver} for results about extremal structure or \cite{ccgmr,godefroysurvey} for the study of norm-attainment of Lipschitz maps).

In this paper paper we will focus on different problems related to octahedral norms in $\Free(M)^{**}$ and in $\Lip(M)$. Recall that the norm of a Banach space $X$, or merely X when there is no ambiguity, is said to be \textit{octahedral} if, for every finite-dimensional subspace $Y$ of $X$ and every $\varepsilon>0$, we can find $x\in S_X$ such that
$$\Vert y+\lambda x\Vert\geq (1-\varepsilon)(\Vert y\Vert+\vert\lambda\vert) $$
holds for every $y\in Y$ and every $\lambda\in\mathbb R$. 


It is classical \cite{blrjfa, deville} that $X$ is octahedral if and only if every convex combination of $w^*$-slices of $B_{X^*}$ has diameter two (this property is known as the \textit{weak-star strong diameter two property ($w^*$-SD2P)}). In the context of free spaces, this characterisation allowed to obtain a full characterisation in \cite{pr} of the class of metric spaces $M$ such that $\mathcal F(M)$ is octahedral. For example, this class contains in particular the unbounded as well as the non-uniformly discrete metric spaces (proved originally in \cite{blrlipschitz}).

A dual version of the above characterisation reads as $X^*$ is octahedral if and only if every convex combination of slices of $B_X$ has diameter two (the \textit{strong diameter two property (SD2P)}, see \cite[Corollary 2.2]{blrjfa}. Hence, it is clear that if $X^{**}$ is octahedral then $X$ is octahedral. However, the converse is not true in general. The natural norm on $\mathcal C([0,1])$ is octahedral but the characteristic function of a singleton is a point of Fr\'echet differentiability of $\mathcal C([0,1])^{**}$. Thus, it is quite natural to ask whether such phenomenon can occur for Lipschitz-free spaces.

\begin{problem}\label{prob:bidualocta}
Is it true for every metric space that if $\mathcal F(M)$ is octahedral, then $\mathcal F(M)^{**}=\Lip(M)^*$ is octahedral? Equivalently, if $\Lip(M)$ has the $w^*$-SD2P, does it necessarily have the SD2P?
\end{problem}

The first main result of this paper (Theorem \ref{theo:theocentralsect2}) shows that the answer is affirmative at least for the unbounded or non-uniformly discrete metric spaces. We also get in Proposition \ref{prop:ivakhnomgeq3} and \ref{prop:discreto} an affirmative answer for a particular family of uniformly discrete bounded metric spaces $(K_m)_m$ which was first considered in \cite[pp. 114]{ivakhno}. In fact we show that for all these metric spaces the space $\Lip(M)$ enjoys the so called \textit{symmetric strong diameter two property (SSD2P)} (see definition in Section \ref{section:bidualoctahedral}) which is in general strictly stronger than the SD2P. It is not known though, whether these two properties coincide for the spaces of Lipschitz functions. It is worth mentioning that these results strengthen (at least formally) the results of \cite{ivakhno} (resp. \cite{hlln}) where it was shown that, in some (resp. all) of the above cases, $\Lip(M)$ enjoys the \textit{slice diameter two property (slice-D2P)} (resp. the $w^*$-SD2P). It is not known whether the slice-D2P implies the SSD2P within the class of spaces of Lispchitz functions, and it is not known whether the SSD2P and the $w^*$-SSD2P coincide in general. On the other hand, it has been recently shown in \cite[Example 3.1]{ostrak} that $w^*$-SSD2P and $w^*$-SD2P are different properties even within the class of spaces of Lipschitz functions. Moreover $w^*$-SSD2P was fully characterised in terms of the metric properties of $M$ in \cite[Theorem 2.1]{ostrak}.

Furthermore, by (a standard) use of the contraction-extension property (CEP) we are able to extend our Theorem \ref{theo:theocentralsect2} to the vector valued case (Theorem \ref{theo:maintheoremLipschitz}).

In the remaining part of the paper we focus on the following problem (see also \cite[Question 3.3]{blrlipschitz}).

\begin{problem}\label{prob:vectorvaluado}
Let $M$ be a metric space such that $\Lip(M)$ is octahedral and let $X$ be a non-trivial Banach space. Is then the space $\Lip(M,X)$ octahedral?
\end{problem}

Note that, thanks to successive papers \cite{gprdaugavet,ikw,am}, $\Lip(M)$ has the Daugavet property if and only if $M$ is a non-trivial length space. In Theorem \ref{teovectorvalSD2P} we solve this problem in the affirmative when $X=Y^*$ for some Banach space $Y$ using the standard fact that in this case we have $\Lip(M,Y^*)=L(Y,\Lip(M))$ isometrically. Let us remark that in general $L(Y,Z)$ can fail to be octahedral even if $Z$ is octahedral \cite[Theorem 3.8]{llr2}. So Theorem 3.1 shows that such an example cannot be produced when $Z$ is an octahedral space of Lispchitz functions. The ideas behind the proof can be also used to improve in Theorem \ref{teotensorgen} the main result of \cite{ruedaquarterly}.
\bigskip

\textbf{Notation:} Given a Banach space $X$, $B_X$ and $S_X$ stand for the closed unit ball and the closed unit sphere, respectively. By a slice of $B_X$ we will mean a set of the form
$$S(B_X,f,\alpha):=\{x\in B_X: f(x)>1-\alpha\},$$
where $f\in S_{X^*}$ and $\alpha>0$. If $X$ is a dual space, say $X=Y^*$, the previous set will be a $w^*$-slice if $f\in Y$. A convex combination of slices of $B_X$ will be a set of the following form
$$\sum_{i=1}^n \lambda_i S_i,$$
where $\lambda_1,\ldots, \lambda_n\in\mathbb R^+$ satisfy that $\sum_{i=1}^n \lambda_i=1$ and $S_i$ is a slice of $B_X$ for every $i\in\{1,\ldots, n\}$. Again, if $X$ is a dual Banach space, the previous set will be a convex combination of $w^*$-slices if each $S_i$ is a $w^*$-slice for every $i\in\{1,\ldots, n\}$.

A Banach space $X$ has the \textit{strong diameter two property (SD2P)} if every convex combination of slices of $B_X$ has diameter $2$. If $X$ is a dual Banach space, then $X$ has the \textit{weak-star strong diameter two property ($w^*$-SD2P)} if every convex combination of $w^*$-star slices of $B_X$ has diameter $2$. As we have pointed out before, the norm of a Banach space $X$ is octahedral if, and only if, $X^*$ has the $w^*$-SD2P \cite[Theorem 2.1]{blrjfa}. Also, $X$ has the SD2P if, and only if, the norm of $X^*$ is octahedral \cite[Corollary 2.2]{blrjfa}.

Given two Banach spaces $X$ and $Y$, we will denote by $L(X,Y)$ the space of operators from $X$ to $Y$, and by $X\pten Y$ the projective tensor product of $X$ and $Y$. We refer to \cite{rya} for a detailed treatment and applications of tensor products.

All the metric spaces considered will be assumed to be complete with no loss of generality. Given a metric space $M$, $B(x,r)$, respectively $S(x,r)$, $\overline{B}(x,r)$, denotes the open ball, respectively the sphere, the closed ball, centered at $x\in M$ with radius $r>0$. Also, given a point $x\in M$ and $0<r<R$, we write
$$C(x,r,R):=\{y\in M: r\leq d(x,y)\leq R\}.$$
We say that $M$ is \textit{uniformly discrete} if $\inf\limits_{x\neq y}d(x,y)>0$. 

Given a metric space $M$ with a designated origin $0$ and a Banach space $X$, we will denote by $\Lip(M,X)$ the Banach space of all $X$-valued Lipschitz functions on $M$ which vanish at $0$ under the standard Lipschitz norm
$$\Vert f\Vert:=\sup\left\{ \frac{\Vert f(x)-f(y)\Vert}{d(x,y)}\ :\ x,y\in M, x\neq y \right\} .$$
First of all, notice that we can consider every point of $M$ as an origin with no loss of generality.

Also, $\Lip(M,X^*)$ is itself a dual Banach space. In fact, the map
$$\begin{array}{ccc}
\delta_{m,x}:\Lip(M,X^*) & \longrightarrow & \mathbb R\\
f & \longmapsto & f(m)(x)
\end{array}$$
defines a linear and bounded map for each $m\in M$ and $x\in X$. In other words, $\delta_{m,x}\in \Lip(M,X^*)^*$. If we define
$$\Free(M,X):=\overline{\spa}(\{\delta_{m,x}\ :\ m\in M, x\in X\}),$$
then we have that $\mathcal F(M,X)^*=\Lip(M,X^*)$. We write $\mathcal F(M):=\mathcal F(M,\mathbb R)$. It is known that $\mathcal F(M,X)=\mathcal F(M)\pten X$ (see \cite[Proposition 2.1]{blrlipschitz}). 
Given a metric space $M$, a Banach space $X$, and a function $f:M\longrightarrow X$, we will denote
$$\supp(f):=\{x\in M:f(x)\neq 0\}.$$
Given a metric space $M$, we say that $M$ is a \emph{length space} if, for every pair of points $x,y \in M$, the distance $d(x,y)$ is equal to the infimum of the lengths of rectifiable curves joining them. Notice that a metric space $M$ is length if, and only if, $\Free(M)$ has the Daugavet property, which in turn is equivalent to the fact that $\Lip(M)$ has the Daugavet property \cite[Theorem 3.5]{gprdaugavet}.

Finally, it is convenient to recall an important tool to construct Lipschitz functions: the classical McShane-Whitney extension theorem. It says that if $N \subseteq M$ and $f \colon N \longrightarrow \mathbb{R}$ is a Lipschitz function, then there is an extension to a Lipschitz function $F \colon M \longrightarrow \mathbb{R}$ with the same Lipschitz constant (see e.g. \cite[Theorem 1.33]{weaver}).

\section{Octahedrality of bidual norms of Lipschitz-free spaces}\label{section:bidualoctahedral}


As we already mentioned in the Introduction, we will be dealing with the following property coming from \cite{anp} which is easily seen to imply the SD2P.

\begin{definition}
Let $X$ be a Banach space. We say that $X$ has the \textit{symmetric strong diameter two property} (\emph{SSD2P}) if, for every $k\in\mathbb N$, every finite family of slices $S_1,\ldots, S_k$ of $B_X$, and every $\varepsilon>0$, there are $x_i\in S_i$ and there exists $\varphi\in B_X$ with $\Vert \varphi\Vert>1-\varepsilon$ such that $x_i\pm \varphi\in S_i$ for every $i\in\{1,\ldots, k\}$.
\end{definition}

If $X$ is a dual Banach space, then the \textit{weak-star symmetric strong diameter two property ($w^*$-SSD2P)} is defined in the natural way just by replacing slices with $w^*$-slices in the above definition.

We can now state the main theorem of this section.

\begin{theorem}\label{theo:theocentralsect2}
Let $M$ be a metric space. If $M$ is unbounded or is not uniformly discrete, then $\Lip(M)$ has the SSD2P. In particular, $\Free(M)^{**}$ is octahedral.
\end{theorem}

The case when $M'$ is infinite was already proved in \cite[Theorems 5.1 and 5.5]{ccgmr}. In spite of being a stronger property, SSD2P is sometimes easier to check than the SD2P. This happens in the case of infinite-dimensional uniform algebras or in Banach spaces with an infinite-dimensional centralizer (c.f. e.g. \cite{hlln} and references therein). It is particularly useful in Banach spaces for which there is not a good description of the dual Banach space (as spaces of Lipschitz functions), because \cite[Theorem 2.1]{hlln} gives a criterion for the SSD2P, which only makes use of weakly convergent nets. Combining this criterion of SSD2P with the sufficient condition for weak convergence of sequences of Lipschitz functions proved in \cite[Lemma 1.5]{ccgmr}, we obtain the following lemma:

\begin{lemma}\label{lemma: sufficient condition of SSD2P}
Let $M$ be a metric space and $X$ be a Banach space. Assume that for every $k\in \mathbb{N}$ and $f_1,\dots,f_k\in S_{\Lip(M,X)}$  there exist sequences $(g_{n,i})\subset \Lip(M,X)$ and $(h_n)\subset S_{\Lip(M,X)}$ such that for every $i\in \{1,\dots,k\}$:
\begin{enumerate}
\item $\limsup\limits_{n\rightarrow \infty} \Vert g_{n,i}\Vert=1$.
    \item $\|g_{n,i}\pm h_n\|\to 1$.
    \item $\supp(g_{n,i}-f_i)\cap \supp(g_{m,i}-f_i)=\emptyset$ whenever $n\neq m$.
    \item $\supp(h_n)\cap \supp(h_m)=\emptyset$ whenever $n\neq m$.
\end{enumerate}
Then $\Lip(M,X)$ has the SSD2P.
\end{lemma}
\begin{proof}
Let $k\in \mathbb{N}$ and $f_1,\dots,f_k\in S_{\Lip(M,X)}$ and assume that $(g_{n,i})\subset \Lip(M,X)$ and $(h_n)\subset S_{\Lip(M,X)}$ satisfy the properties (1)-(4) above.

By \cite[Theorem 2.1]{hlln}, it suffices to show that there are $(\psi_{n,i})\subset S_{\Lip(M,X)}$ and $(\varphi_n)\subset S_{\Lip(M,X)}$ such that $\psi_{n,i}\to f_i$ weakly, $\varphi_n\to 0$ weakly, and $\|\psi_{n,i}\pm \varphi_n\|\to 1$ for every $i\in \{1,\dots,k\}$.

Note that we can take $\varphi_n:=h_n$ for every $n$, because $(h_n)$ is a bounded sequence with disjoint supports and by \cite[Lemma 1.5]{ccgmr} we get that it converges weakly to 0. A similar argument gives us that $(g_{n,i})$ converges weakly to $f_i$ for every $i\in \{1,\dots,k\}$. Moreover, by the weak lower semicontinuity of the norm, we get that $\|g_{n,i}\|\to 1$. Finally, the choice of $\psi_{n,i}:=\frac{g_{n,i}}{\|g_{n,i}\|}$ finishes the proof. 
\end{proof}

As an application of the previous lemma we get a sufficient condition on a metric space $M$ to ensure that $\Lip(M)$ has the SSD2P.

\begin{lemma}\label{lemma:technicallip}
Let $M$ be a metric space. Assume that there are six sequences of scalars $(r_n), (s_n), (t_n),(T_n),(S_n)$, and $(R_n)$ satisfying the following properties:
\begin{enumerate}
    \item The following inequalities 
    $$r_n<s_n<t_n<T_n<S_n<R_n$$
    hold for every $n\in\mathbb N$.
    \item The set $C(0,t_n,T_n)$ is non-empty for every $n\in\mathbb N$.
    \item $C(0,r_n,R_n)\cap C(0,r_m,R_m)=\emptyset$ if $n\neq m$.
    \item The sequences $( \frac{r_n}{s_n})$,$( \frac{s_n}{t_n})$,$( \frac{t_n}{T_n})$,$( \frac{T_n}{S_n})$ and $( \frac{S_n}{R_n})$ converge to $0$.
\end{enumerate}
Then the space $\Lip(M)$ has the SSD2P.
\end{lemma}

\begin{proof}
Let $k\in\mathbb{N}$ and $f_1,\dots,f_k\in S_{\Lip(M)}$, and let us apply Lemma \ref{lemma: sufficient condition of SSD2P}.

Define functions $g_{n,i}\colon B(0,r_n)\cup C(0,s_n,S_n)\cup M\setminus \overline{B}(0,R_n)\to \mathbb{R}$ by
\[
g_{n,i}(x):=
\begin{cases}
  f_i(x), & \text{if } d(0,x)<r_n\\
  0, & \text{if } s_n<d(0,x)<S_n\\
  f_i(x), & \text{if } d(0,x)>R_n.
\end{cases}
\]

Let us estimate $\|g_{n,i}\|$.  Let $x,y\in M$ and $x\neq y$. Denote by
\[
A:=|g_{n,i}(x)-g_{n,i}(y)|.
\] 
We only have to consider two non-trivial cases, otherwise it is clear that $A=0$ or $A\leq d(x,y)$:
\begin{itemize}
    \item[(a)] If $d(0,x)<r_n$ and $d(0,y)\in (s_n, S_n)$, then
    \begin{align*}
        \frac{A}{d(x,y)}&=\frac{|f_i(x)|}{d(x,y)}\leq\frac{d(0,x)}{d(x,y)}\leq\frac{r_n}{d(0,y)-d(0,x)}\\
        &\leq\frac{r_n}{s_n-r_n}=\frac{1}{\frac{s_n}{r_n}-1}.
    \end{align*}
    \item[(b)] If $d(0,x)\in (s_n, S_n)$ and $d(0,y)>R_n$, then
    \begin{align*}
        \frac{A}{d(x,y)}&=\frac{|f_i(y)|}{d(x,y)}\leq\frac{d(0,y)}{d(x,y)}\leq\frac{d(0,x)+d(x,y)}{d(x,y)}\\
        &\leq1+\frac{S_n}{R_n-S_n}=1+\frac{1}{\frac{R_n}{S_n}-1}.
    \end{align*}
\end{itemize}
Thus $\|g_{n,i}\|\leq \max\{1+\frac{1}{\frac{R_n}{S_n}-1},\frac{1}{\frac{s_n}{r_n}-1}\}$. Extend now $g_{n,i}$ norm preservingly to $M$, and note that the previous inequality proves that the sequence $(g_{n,i})$ satisfies (1) in Lemma \ref{lemma: sufficient condition of SSD2P}. Note also that $\supp(f_i-g_{n,i})\subset C(0,r_n,R_n)$, so condition (3) of Lemma \ref{lemma: sufficient condition of SSD2P} is satisfied.

Pick a sequence of functions $h_n\in S_{\Lip (M)}$ such that $\supp(h_n)\subset C(0,t_n,T_n)$. In order to apply Lemma \ref{lemma: sufficient condition of SSD2P} let us estimate $\|g_{n,i}\pm h_n\|$. Let $x,y\in M$ and $x\neq y$. Denote by
\[
B:=|(g_{n,i}\pm h_n)(x)-(g_{n,i}\pm h_n)(y)|.
\]
Observe that again we have two non-trivial cases, otherwise the inequality $B\leq d(x,y)$ obviously holds:
\begin{itemize}
    \item[(a)] If $x\in C(0,t_n,T_n)$ and $y\in C(0,s_n, S_n)$, then $g_{n,i}(x)=g_{n,i}(y)=0$ and so
    \begin{align*}
        \frac{B}{d(x,y)}=\frac{|h_n(x)-h_n(y)|}{d(x,y)}\leq\|h_n\|=1.
    \end{align*}
    \item[(b)] If $x\in C(0,t_n,T_n)$ and $y\notin C(0,s_n, S_n)$, then $h_n(y)=0$ and so
      \begin{align*}
        \frac{B}{d(x,y)}&\leq \frac{|g_{n,i}(x)-g_{n,i}(y)|}{d(x,y)}+\frac{|h_{n}(x)-h_{n}(y)|}{d(x,y)}\\
        &\leq \begin{cases}
  \frac{|g_{n,i}(y)|}{d(x,y)}+1, & \text{if } d(0,y)<s_n\\
  1+\frac{|h_n(x)|}{d(x,y)}, & \text{if } d(0,y)>S_n.
\end{cases}\\
        &\leq 1+\begin{cases}
  \frac{d(0,y)}{d(x,0)-d(y,0)}\leq \frac{1}{\frac{t_n}{s_n}-1}, & \text{if } d(0,y)<s_n\\
  \frac{d(0,x)}{d(0,y)-d(0,x)}\leq \frac{1}{\frac{S_n}{T_n}-1}, & \text{if } d(0,y)>S_n.
\end{cases}
        \end{align*}
\end{itemize}
Therefore $\|g_{n,i}\pm h_n\|\leq\max \{1+\frac{1}{\frac{t_n}{s_n}-1},1+\frac{1}{\frac{S_n}{T_n}-1}\}$ from where $\Vert g_{n,i}\pm h_n\Vert\rightarrow 1$. Now Lemma \ref{lemma: sufficient condition of SSD2P} implies that $\Lip(M)$ has the SSD2P, so we are done.
\end{proof}

\begin{proof}[Proof of Theorem \ref{theo:theocentralsect2}]
If $M$ is discrete but not uniformly discrete then the result follows from \cite[Theorem 5.4]{ccgmr}. For the remaining cases we will apply Lemma \ref{lemma:technicallip}. In order to do so, if $M'\neq \emptyset$ (we assume with no loss of generality that $0\in M'$), we can inductively construct the sequences $(r_n), (s_n), (t_n),(T_n),(S_n)$, and $(R_n)$ such that $\max\{\frac{r_n}{s_n},\frac{s_n}{t_n},\frac{t_n}{T_n},\frac{T_n}{S_n},\frac{S_n}{R_n}\}<\frac{1}{2^n}$, that $C(0,t_n,T_n)\neq \emptyset$ and such that $R_{n+1}<r_n$. On the other hand, if $M'=\emptyset$ then the remaining case is when $M$ is unbounded. In such a case, the sequences can be constructed just in a similar way but imposing $R_n<r_{n+1}$.\end{proof}

We are now going to prove a vector valued version of Theorem \ref{theo:theocentralsect2}. For this, let us introduce some notation. We recall that given a metric space $M$ and a Banach space $X$, it is said that the pair $(M,X)$ satisfies the \emph{contraction-extension property} (\emph{CEP}) if McShane-Whitney's extension theorem holds for $X$-valued Lipschitz functions from subsets of $M$, that is, given $N\subseteq M$ and a Lipschitz function $f\colon N\longrightarrow X$, there exists a Lipschitz function $F\colon M\longrightarrow X$ which extends $f$ and satisfies that
$$\Vert F\Vert_{\Lip(M,X)}=\Vert f\Vert_{\Lip(N,X)}.$$
\begin{remark}
Notice that, in the case that the codomain space is a dual space, then the CEP has a reformulation in the language of tensor product spaces: given a metric space $M$ and a Banach space $X$, then the pair $(M,X^*)$ has the CEP if, and only if, for every subset $N$ of $M$ we have that $\Free(N)\pten X$ is an canonical (isometric) subspace of $\Free(M)\pten X$. This follows taking into account the isometric isomorphism between $\Lip(M,X^*)$ and $L(\Free(M),X^*)$ together with \cite[Corollary 2.12]{rya}.
\end{remark}

When $M$ is a Banach space then the definition of the CEP given above agrees with \cite[Definition 2.10]{beli}. The standard examples of pairs having the CEP are $(M,\mathbb R)$ or more generally $(M,\ell_\infty(\Gamma))$ for any metric space $M$ and any set $\Gamma$, and $(H,H)$ for any Hilbert space $H$. Furthermore, it is clear that if a pair $(M,X)$ has the CEP then the pair $(N,X)$ has the CEP for every subset $N$ of $M$. Notice also that the CEP is in fact a very restrictive property (see \cite[Theorem 2.11]{beli}).

Now the vector valued version of Lemma \ref{lemma:technicallip} holds as can be seen if one replaces the use of McShane-Whitney theorem with an application of the CEP in the proof. Thus using this modified Lemma \ref{lemma:technicallip} and using \cite[Theorem 5.8]{ccgmr} instead of \cite[Theorem 5.4]{ccgmr} we can prove the following vector valued version of Theorem \ref{theo:theocentralsect2}.

\begin{theorem}\label{theo:maintheoremLipschitz}
Let $M$ be a metric space and $X$ be a Banach space such that the pair $(M,X)$ has the CEP. If $M$ is unbounded or is not uniformly discrete, then $\Lip(M,X)$ has the SSD2P. In particular, if $X=Y^*$ for some Banach space $Y$, we additionally have that $\mathcal F(M,Y)^{**}$ is octahedral.
\end{theorem}

Notice that this theorem improves \cite[Theorem 2.4]{blrlipschitz}, where the $w^*$-SD2P is obtained under the same assumptions, and it gives a positive answer to \cite[Question 3.1]{blrlipschitz}.

In order to answer Problem \ref{prob:bidualocta} it remains to deal with the bounded uniformly discrete metric spaces. We can still hope that in this setting the $w^*$-SD2P in $\Lip(M)$ implies the SD2P but we can not expect that it would imply the SSD2P. Indeed, an example appears in \cite[Example 3.1]{ostrak} of a bounded uniformly discrete metric space such that $\Lip(M)$ has the $w^*$-SD2P but fails the $w^*$-SSD2P. In this context, we can ask whether the $w^*$-SSD2P in $\Lip(M)$ implies the SSD2P but notice that this question is even open for general dual Banach spaces.

The following particular family of uniformly discrete metric spaces was put in focus in \cite{ivakhno} and later studied in \cite{hlln}.

For $m\in \mathbb{N}$ denote by
\[
  K_m:=\{x \in \ell_\infty \colon x(k) \in \{0,1,\dots,m\} \;
  \mbox{for all} \; k \in \mathbb{N}\}
\]
with metric inherited from $\ell_\infty$. Y. Ivakhno proved in \cite{ivakhno} that $\Lip(K_m)$ satisfies the slice-D2P if $m\in\{1,2\}$. Also, it was posed as an open question \cite[Question]{ivakhno} whether $\Lip(K_m)$ has the slice-D2P for $m\geq 3$. In \cite[Theorem 5.7]{hlln}, it was proved that $\Lip (K_m)$ has the w$^\ast$-SSD2P for every $m\in \mathbb{N}$. Adapting the ideas in \cite[Proposition 5.6]{hlln} we will prove the following Proposition. In particular, this gives a positive answer to Ivakhno's question.

\begin{proposition}\label{prop:ivakhnomgeq3} If $m\geq 3$, then the Banach space $\Lip (K_m)$ has the SSD2P.
\end{proposition}
\begin{proof}
Assume that $m\geq 3$ and $f_1,\dots,f_k\in S_{\Lip(K_m)}$, and let us apply Lemma \ref{lemma: sufficient condition of SSD2P}. Consider the sequence $(u_n)\in K_m$ such that 
$u_n(l)=m$ if $n=l$ and $u_n(l)=0$ otherwise.

For every $n$ and $i$, let
\[
a_{n,i}:=\frac12\left(\inf_{x\in  S(u_n,2)}f_i(x)+ \sup_{x\in S(u_n,2)}f_i(x)\right).
\]
Let $n\in \mathbb{N}$. Note that if $x,y\in S(u_n,2)$, then $d(x,y)\leq 2$. Therefore, for every $z\in S(u_n,2)$ we have that $|a_{n,i}-f_i(z)|\leq 1$.
Define
\[
g_{n,i}(x):=
\begin{cases}
  a_{n,i}, & \text{if } x\in \overline{B}(u_n,1)\\
  f_i & \text{elsewhere}.
\end{cases}
\]
Then $\|g_{n,i}\|\leq 1$ for every $n$ and $i$. Indeed, fix $n$ and $i$. Let $x \in \overline{B}(u_n,1)$ and $y \in K_m \setminus \overline{B}(u_n,1)$.
\begin{itemize}
    \item[(a)] If $d(u_n,y)=2$, then
  \[
    |g_{n,i} (x) - g_{n,i}(y)|
    =
    |a_{n,i} - f_i(y)| \leq 1\leq d(x,y).
  \]
  \item[(b)] If $d(u_n,y)\geq 3$, then find $z\in S(u_n,2)$ such that $d(u_n,y)=d(u_n,z)+d(z,y)$. Thus
 \begin{align*}
    |g_{n,i} (x) - g_{n,i}(y)|
    &=
      |a_{n,i} - f_i(y)| \leq
      |a_{n,i}-f_i(z)|+ |f_i(z)- f_{i}(y)| \\
    &\leq
      1 + d(z,y)=1+d(u_n,y)-2\\
      &=d(u_n,y)-1\leq d(u_n,y)-d(u_n,x)\\
      &\leq d(x,y).
  \end{align*}
\end{itemize}

Observe that $\supp(g_{n,i}-f_i)\subset \overline{B}(u_n,1)$ for every $n$ and $i$, and $\overline{B}(u_r,1)\cap \overline{B}(u_s,1)=\emptyset$ if $r\neq s$.

Define now
\[
h_n(x):=
\begin{cases}
  1, & \text{if } x=u_n,\\
  0 & \text{elsewhere}.
\end{cases}
\]
Clearly $\|h_n\|=1$ for every $n\in \mathbb{N}$ and $\supp(h_n)=\{u_n\}$.

Finally, in order to apply Lemma \ref{lemma: sufficient condition of SSD2P}, let us verify that $\|g_{n,i} \pm h_n\| \leq 1$ for every $n$ and $i$. Fix $n$ and $i$. Let $x \in \overline{B}(u_n,1)$ and $y \in K_m \setminus \overline{B}(u_n,1)$.
\begin{itemize}
    \item[(a)] If $x \not= u_n$, then $h_n(x) = h_n(y) = 0$ and therefore
  \[
    |(g_{n,i} \pm h_n)(x) - (g_{n,i} \pm h_n)(y)|
    =
    |g_{n,i}(x) - g_{n,i}(y)| \leq d(x,y).
  \]
    \item[(b)]   If $x = u_n$, then find $z \in S(u_n,1)$ such that
  $d(u_n,y) = d(u_n,z) + d(z,y)$.
  Thus
  \begin{align*}
    |(g_{n,i} \pm h_n)(u_n) - (g_{n,i} \pm h_n)(y)|
    &=
      |g_{n,i}(z) \pm h_n(u_n) - g_{n,i}(y)| \\
    &\leq
      |g_{n,i}(z) - g_{n,i}(y)| + 1 \\
    &\leq
      d(z,y) + d(u_n,z) = d(u_n,y).
  \end{align*}
\end{itemize}
  Taking supremum we get that $\Vert g_{n,i}\pm h_n\Vert\leq 1$, so Lemma \ref{lemma: sufficient condition of SSD2P} applies to get that $\Lip(M)$ has the SSD2P. \end{proof}

Notice that the previous proof does not work in the case $m=2$ because $K_2=\overline{B}(e,1)$, where $e=(1,1,1,\dots)$. Thus we do not know whether Proposition \ref{prop:ivakhnomgeq3} holds for $m=2$ nor if $\Lip(K_2)$ has at least the SD2P. Notice that $\Lip(K_2)$ has the slice-D2P \cite{ivakhno}. However, a Banach space $X$ can have the slice-D2P but contain convex combinations of slices of arbitrarily small diameter \cite{blradvances}. On the other hand, the case $m=1$ is part of the following general result.

\begin{proposition}\label{prop:discreto}
If $M$ is an infinite set with the discrete metric, then $\Lip (M)$ has the SSD2P.
\end{proposition}
\begin{proof}
Assume with no loss of generality that $d(x,y)=1$ for every $x\neq y$. Pick $f_1,\dots,f_k\in S_{\Lip(M)}$, and let us apply Lemma \ref{lemma: sufficient condition of SSD2P}. Denote by
\[
a_i:=\frac12\left(\inf_{x\in M}f_i(x)+ \sup_{x\in M}f_i(x)\right)\quad \text{for every $i\in\{1,\dots,k\}$}.
\]
Fix a disjoint sequence $x_1, y_1, x_2, y_2, \dots\in M$.
Define
\[
g_{n,i}(x):=
\begin{cases}
  a_{i}, & \text{if } x\in \{x_n,y_n\}\\
  f_i & \text{elsewhere},
\end{cases}
\]
and
\[
h_n(x):=
\begin{cases}
  \frac12, & \text{if } x=x_n,\\
  -\frac12, & \text{if } x=y_n,\\
  0 & \text{elsewhere}.
\end{cases}
\]
Observe that $\|g_{n,i}\|, \|h_n\|\leq 1$ for every $n$ and $i$. Note that $\supp(g_{n,i}-f_i)\subset\supp(h_n)=\{x_n,y_n\}$ so, in order to apply Lemma \ref{lemma: sufficient condition of SSD2P}, let us verify that $\|g_{n,i} \pm h_n\|\leq 1$ holds for every $n$ and $i$. Fix $n$ and $i$ and let $x,y\in M$, $x\neq y$. Suppose that $x=x_n$ and $y\neq y_n$ because the remaining cases are trivial or similar. Note that, since $\Vert f_i\Vert= 1$, then $\sup_{x\in M}f_i(x)-\inf_{x\in M}f_i(x)=1$ holds for every $i\in\{1,\dots,k\}$. Thus
\begin{align*}
    |(g_{n,i}&\pm h_n)(x)- (g_{n,i}\pm h_n)(y)|=\left |a_i \pm\frac12 - f_i(y)\right |\\
    &=\left|\frac12\left(\inf_{x\in M}f_i(x)+ \sup_{x\in M}f_i(x)\right)\pm \frac12\left(\sup_{x\in M}f_i(x)-\inf_{x\in M}f_i(x)\right)-f_i(y)\right|\\
    &\leq \|f_i\|= 1=d(x,y). 
\end{align*}
The arbitrariness of $x$ and $y$ yields that $\Vert g_{n,i}\pm h_n\Vert\leq 1$, so Lemma \ref{lemma: sufficient condition of SSD2P} applies.
\end{proof}

\section{Octahedrality of dual norms of Lipschitz-free spaces }\label{section:d2pfreespaces}

In this section we will deal with a partial positive answer to Problem \ref{prob:vectorvaluado}, namely the theorem~\ref{teovectorvalSD2P} below.

\begin{theorem}\label{teovectorvalSD2P}
Let $M$ be a metric space and $X$ be a Banach space. If the norm of $\Lip(M)$ is octahedral, then the norm of $\Lip(M,X^*)$ is octahedral.
\end{theorem}

The proof of this theorem depends on the canonical identification \\ $\Lip(M,X^*)=L(X,\Lip(M))$ and is strongly inspired by the proof of the following result which improves \cite[Theorem 2.1]{ruedaquarterly}, where the space $Y$ below was additionally assumed to have a monotone basis and be isometrically a subspace of $L_1$.

\begin{theorem}\label{teotensorgen}
Let $X$ and $Y$ be Banach spaces. If $Y$ is finitely representable in $\ell_1$ and has the MAP and the norm of $X$ is octahedral, then the operator norm of $H$ is octahedral whenever $H$ is a subspace of $L(Y,X)$ containing the space of finite-rank operators.
\end{theorem}

\begin{proof}
Pick $T_1,\ldots, T_n\in S_H$ and $\varepsilon>0$, and let us find $T\in H$ with $\Vert T\Vert\leq 1$ such that $\Vert T_i+T\Vert>2-\varepsilon$ holds for every $i\in\{1,\ldots, n\}$. This is enough in view of \cite[Proposition 2.1]{hlp2}. To this end, consider $y_i\in S_{Y}$ such that $\Vert T_i(y_i)\Vert>1-\frac{\varepsilon}{4}$ for every $i\in\{1,\ldots, n\}$. Since $Y$ has the MAP we can find a finite-rank operator $P:Y\longrightarrow Y$ with $\Vert P\Vert\leq 1$ and such that $\Vert P(y_i)\Vert>1-\frac{\varepsilon}{4}$. Now the rest of the proof follows word-by-word the proof of \cite[Theorem 2.1]{ruedaquarterly}. However, let us give an sketch of the proof in order to motivate the ideas behind the proof of Theorem \ref{teovectorvalSD2P}.

Define $V:=\spa\{P(y_i): 1\leq i\leq n\}$. Since $Y$ is finitely representable in $\ell_1$ then we can find $\varphi:V\longrightarrow \ell_1$ such that $\Vert \varphi\Vert\leq 1$ and such that $\Vert \varphi(P(y_i))\Vert>1-\frac{\varepsilon}{4}$ holds for every $i\in\{1,\ldots, n\}$. Denote, for every natural number $p$, by $Q_p:\ell_1\longrightarrow \ell_1$ the projection onto the first $p$ coordinates. By the monotonicity of the $\ell_1$ norm we can find $p\in\mathbb N$ such that $\Vert Q_p(\varphi(P(y_i)))\Vert>1-\frac{\varepsilon}{4}$ holds for every $i\in\{1,\ldots, n\}$.

Define now $F:=\spa\{T_i(y_i):1\leq i\leq n\}$, which is a finite-dimensional subspace of $X$. Since the norm of $X$ is octahedral we can find $z_1,\ldots, z_p\in S_X$ such that:
\begin{enumerate}
\item\label{generallr21} $\Vert f+z\Vert>\left(1-\frac{\varepsilon}{4}\right)(\Vert f\Vert+\Vert z\Vert)$ holds for all $f\in F$ and every $z\in Z:= \spa\{z_1,\ldots, z_p\}$.
\item \label{generallr22} $Z$ is $(1+\frac{\varepsilon}{4})$-isometric to $Q_p(\ell_1)$.
\end{enumerate}

From (\ref{generallr22}) we can consider a norm-one map $\Phi:Q_p(\ell_1)\longrightarrow Z$ such that $\Vert \Phi(Q_p(\varphi(P_k(y_i))))\Vert>1-\frac{\varepsilon}{4}$.
Define finally $T:=i\circ\Phi\circ Q_p\circ \varphi\circ P:Y\longrightarrow X$, where $i:Z\hookrightarrow X$ denotes the inclusion operator. Since $P$ is a finite-rank operator we conclude that $T$ is a finite-rank operator and then $T\in H$. Moreover
$$\Vert T\Vert\leq \Vert i\Vert\Vert \Phi\Vert\Vert Q_p\Vert\Vert\varphi\Vert\Vert P\Vert\leq 1.$$
Finally notice that $\Vert T(y_i)\Vert >1-\frac{\varepsilon}{4}$ holds for every $i\in\{1,\ldots, n\}$. Moreover, since $T(Y)\subseteq Z$, we get from (\ref{generallr21}) that
\[\begin{split}\Vert T_i+T\Vert& \geq \Vert T_i(y_i)+T(y_i)\Vert> \left( 1-\frac{\varepsilon}{4} \right)(\Vert T_i(y_i)\Vert+\Vert T(y_i)\Vert)\\
& >\left(
1-\frac{\varepsilon}{4}\right)\left(1-\frac{\varepsilon}{4}+1-\frac{\varepsilon}{4} \right)>2-\varepsilon.
\end{split}
\]
Since $\varepsilon>0$ was arbitrary we conclude the desired result. \end{proof}

Before giving the proof of Theorem \ref{teovectorvalSD2P} let us outline the underlying ideas, which are strongly based on those of the proof of Theorem \ref{teotensorgen}. In that proof, the key idea was to take advantage of the MAP assumption to construct, given a finite-dimensional subspace $E$ of $X$, a norm one operator $T:X\longrightarrow \ell_1$ such that $\Vert T(e)\Vert\geq (1-\varepsilon)\Vert e\Vert$ holds for every $e\in E$, and then make use of the fact that, in a Banach space $Y$ whose norm is octahedral, for every finite-dimensional subspace $Y_0\subseteq Y$ we can find a subspace $Z\subseteq Y$ which is $(1+\varepsilon)$-isometric to $\ell_1$ and satisfying that
$$\Vert y+z\Vert>(1-\varepsilon)(\Vert y\Vert+\Vert z\Vert)$$
holds for every $y\in Y_0$ and $z\in Z$. Now we will prove that, under the hypothesis of Theorem 3.1, for every finite-dimensional subspace $Y_0\subseteq \Lip(M)$ we can find a subspace $Z\subseteq \Lip(M)$ which is isometric to $c_0$ and satisfying that
$$\Vert y+z\Vert>(1-\varepsilon)(\Vert y\Vert+\Vert z\Vert)$$
holds for every $y\in Y_0$ and $z\in Z$. After that, we will make use of theory of finite representability and extension of operators in order to get rid of the MAP assumption.


\begin{proposition}\label{prop:octaLipc0}
Let $M$ be a complete length metric space. Then, for all $f_1,\ldots, f_k\in S_{\Lip(M)}$ and every $\varepsilon>0$ there exists a closed subspace $Y\subseteq \Lip(M)$ such that
\begin{itemize}
    \item $Y$ is isometric to $c_0$ and
    \item for all $1\leq i\leq k$ and all $f\in S_Y$ we have $\Vert f_i+f\Vert\geq 2-\varepsilon$.
\end{itemize}
\end{proposition}

\begin{proof}
Let $\varepsilon>0$ and $f_1,\ldots, f_k\in S_{\Lip(M)}$ be given and let us assume, as we may, that for every $1\leq i\leq k$ there is $1\leq j\leq k$ such that $f_i=-f_j$. Recall that it has been proved in \cite{ikw} that any length space is spreadingly local. In particular, the set
$$V_i:=\left\{x\in M: \inf\limits_{\beta>0}\Vert f_i\upharpoonright_{B(x,\beta)} \Vert>1-\varepsilon \right\}$$
is infinite for every $1\leq i\leq k$. A simple induction on $k$ shows that we may select countably infinite $W_i\subseteq V_i$ such that the family $\{W_i: 1\leq i\leq k\}$ is pairwise disjoint. We may further assume that each $W_i$ admits at most one cluster point, which, if it exists, does not belong to $W_j$ for any $1\leq j\leq k$. Let us write $W_i:=\{w_n^i: n\in\mathbb N\}$. By the properties of the sets $W_i$ it is clear that for every $n\in\mathbb N$ there is an $r_n>0$ such that $d(w_n^i,w_m^j)>2r_n$ for every $1\leq i\leq k$ and every $(m,j)\neq (n,i)$. Obviously, we may assume that the sequence $(r_n)$ is decreasing. Let us fix $n$. By the properties of $W_i$, for every $1\leq i\leq k$ there are points $x_n^i, y_n^i\in B(w_n^i,\frac{r_n}{8})$ such that $\frac{f_i(x_n^i)-f(y_n^i)}{d(x_n^i, y_n^i)}>1-\varepsilon$. We define
$$g_n(x):=\left\{\begin{array}{cc}
0,     & \mbox{if }x\in M\setminus \bigcup\limits_{i=1}^k B(w_n^i,\frac{r_n}{2}),  \\
0,    & \mbox{if } x=y_n^i, 1\leq i\leq k,\\
d(x_n^i,y_n^i) & \mbox{if } x=x_n^i, 1\leq i\leq k.
\end{array} \right.$$
It is clear that $\Vert g_n\Vert=1$. We extend $g_n$ in a norm-preserving way to the whole $M$ using the McShane-Whitney theorem. We denote the extension by $g_n$ too. Obviously $\supp(g_n)\subset \bigcup\limits_{i=1}^k \overline{B}(w_n^i,\frac{r_n}{2})$. We set $Y:=\overline{\operatorname{span}}\{g_n\}$. We claim that $(g_n)$ is isometrically equivalent to the canonical basis of $c_0$. Indeed, for $m<n$ let $x\in \supp(g_m)$ and $y\in \supp(g_n)$. Then $d(x,y)\geq r_m-\frac{r_n}{2}-\frac{r_m}{2}>0$. Since $M$ is a length metric space (and thus local), \cite[Lemma 3.4]{kms} yields the conclusion.

To prove the remaining parts, pick $g\in S_Y$ and let us prove that $\Vert f_i+g\Vert>2-\varepsilon$ for all $1\leq i\leq k$. To this end, take $\eta>0$ and find $p\in\mathbb N$ and $\lambda_1,\ldots, \lambda_p\in \mathbb R$ such that 
$$\left\Vert g-\sum_{j=1}^p \lambda_j g_j \right\Vert<\eta.$$
Define $f:=\sum_{j=1}^p \lambda_j g_j$ and notice that $\Vert f\Vert=\max\limits_{1\leq j\leq p}\vert \lambda_j\vert>1-\eta$. Now select $j$ so that $\vert \lambda_j\vert>1-\eta$. If $\lambda_j>1-\eta$, we have for every $1\leq i\leq k$ that
$$\Vert f_i+f\Vert\geq \frac{f_i(x_j^i)-f_i(y_j^i)+\lambda_j(g_j(x_j^i)-g_j(y_j^i))}{d(x_j^i,y_j^i)}>1-\varepsilon+\lambda_j>2-\varepsilon-\eta$$
since $f(x_j^i)=\lambda_j g_j(x_j^i)$ and $f(y_j^i)=\lambda_j g_j(y_j^i)$. If $\lambda_j<-1+\eta$, for $1\leq i\leq k$ let $1\leq l\leq k$ be such that $f_i=-f_l$. Then
$$\Vert f_i+f\Vert=\Vert f_l-f\Vert$$
which is larger than $2-\varepsilon-\eta$ by the preceding case. Finally, since $\Vert g-f\Vert<\eta$ and since $\eta>0$ was arbitrary we get that $\Vert f_i+g\Vert\geq 2-\varepsilon$ as desired.
\end{proof}

Let us now go on into details and let us prove that the norm of $\mathcal F(M,X)^*=\Lip(M,X^*)=L(\mathcal F(M),X^*)=L(X,\Lip(M))$ is octahedral.

\begin{proof}[Proof of Theorem \ref{teovectorvalSD2P}]
Pick $T_1,\ldots, T_k\in S_{L(X,\Lip(M))}$ and $\varepsilon>0$, and let us find a norm-one operator $T:X\longrightarrow \Lip(M)$ such that
\begin{equation}\label{ecuateovectorobjetivo}
\Vert T_i+T\Vert>2-\varepsilon
\end{equation}
holds for every $i\in\{1,\ldots, k\}$. For every $i\in\{1,\ldots, k\}$, choose $a_i\in S_X$ such that $f_i:=T_i(a_i)$ has norm greater than $1-\delta^2$, for $0<\delta <\varepsilon$.

Pick $\eta>0$ such that $\eta+\delta<\varepsilon$ and let us construct an operator $\phi:X\longrightarrow c_0$ such that $\Vert \phi(a_i)\Vert>1-\eta$ holds for every $i\in\{1,\ldots, k\}$. Since $X$ is finitely representable in $c_0$ (c.f. e.g. \cite[Example 11.1.2]{alka}) let an operator $\phi:\spa\{a_1,\ldots,a_k\}\longrightarrow c_0$ be such that $\Vert \phi(a_i)\Vert>1-\eta$ holds for every $i\in\{1,\ldots, k\}$ and that $\Vert \phi\Vert<1$. Since $c_0$ is an $L_1$-predual, the Lindenstrauss celebrated theorem \cite[Theorem 6.1]{linds} yields a compact extension (still denoted by $\phi$) $\phi:X\longrightarrow c_0$ such that $\Vert \phi\Vert\leq 1$. We apply Proposition \ref{prop:octaLipc0} to $\frac{f_i}{\Vert f_i\Vert}$, $1\leq i\leq k$, to find a sequence $(g_n)$ which is isometrically isomorphic to the canonical basis of $c_0$ such that 
$$\Vert f_i+f\Vert\geq 2-\delta$$
holds for every $1\leq i\leq k$ and every $f\in S_Y$.

Since $(g_n)$ is isometrically equivalent to the $c_0$ basis let $\Psi:c_0\longrightarrow \overline{\spa}\{g_n\}$ be the canonical linear isometry given by $\Psi(e_n):=g_n$. Call $j:\overline{\spa}\{g_n\}\hookrightarrow \Lip(M)$ the canonical inclusion and define
$$T:=j\circ \Psi\circ \phi:X\longrightarrow \Lip(M).$$
Notice that $\Vert T\Vert\leq \Vert j\Vert \vert \Psi\Vert \Vert \phi\Vert\leq 1$. Furthermore, since $j$ and $\Psi$ are isometries we get that $\Vert T(a_i)\Vert>1-\eta$ holds for every $i\in\{1,\ldots, k\}$. Also, since $T(X)\subseteq \overline{\spa}\{g_n\}$ we get that 
$$ \left\Vert T_i(a_i)+\frac{T(a_i)}{\Vert T(a_i)\Vert} \right\Vert=\left\Vert f_i+\frac{T(a_i)}{\Vert T(a_i)\Vert} \right\Vert\geq 2-\delta$$
by the condition on the sequence $(g_n)$, from where $\Vert T_i+T\Vert\geq \Vert T_i(a_i)+T(a_i)\Vert\geq 2-\delta-\eta>2-\varepsilon$. This proves that $T$ satisfies \eqref{ecuateovectorobjetivo} and finishes the proof.\end{proof}

\begin{remark}\label{remark:subspacesoh}
For every $n\in\mathbb N$ denote by $P_n:c_0\longrightarrow c_0$ the projection onto the first $n$ coordinates. Notice that, in the proof above, we can find $n$ large enough so that $\Vert P_n(\phi(a_i))\Vert>1-\eta$. This means that the finite-rank operator $\bar T=j\circ\Psi\circ P_n\circ \phi$ satisfies that $\Vert T_i+\bar T\Vert>2-\varepsilon$. From here, it can be proved that the operator norm on $H$ is octahedral for every subspace $H\subseteq L(X,\Lip(M))$ containing finite-rank operators.
\end{remark}

As a consequence we get the following corollary, which answers Question 3.3 in \cite{blrlipschitz}.

\begin{corollary}\label{corocarasd2pvectorval}
Let $M$ be a metric space. The following assertions are equivalent:
\begin{enumerate}
    \item\label{corocarasd2pvectorval1} $\mathcal F(M,X)$ has the SD2P for every Banach space $X$.
    \item\label{corocarasd2pvectorval2} $\mathcal F(M)$ has the SD2P.
    \item\label{corocarasd2pvectorval3} $\mathcal F(M)$ does not have any strongly exposed point.
\end{enumerate}
\end{corollary}

\begin{proof}
\eqref{corocarasd2pvectorval1}$\Rightarrow$\eqref{corocarasd2pvectorval2} is trivial, whereas \eqref{corocarasd2pvectorval2}$\Rightarrow$\eqref{corocarasd2pvectorval1} follows from Theorem \ref{teovectorvalSD2P} and by \cite[Corollary 2.2]{blrjfa}. Finally, the equivalence of \eqref{corocarasd2pvectorval2} with \eqref{corocarasd2pvectorval3} is \cite[Theorem 1.5]{am}.
\end{proof}

\begin{remark}
Given a metric space $M$, the existence of a Banach space $X$ such that $\mathcal F(M,X)$ has the SD2P does not imply that $\mathcal F(M)$ has the SD2P. Indeed, given $X=L_1([0,1])$, it follows that
$$\mathcal F(M,X)=\mathcal F(M)\pten X=\mathcal F(M)\pten L_1([0,1])=L_1([0,1],\mathcal F(M))$$
has the SD2P (even the Daugavet property) for every metric space $M$ (see the Example after Corollary 2.3 in \cite{werner}).
\end{remark}

\begin{remark}Let $X$ and $Y$ be two Banach spaces. In \cite[Question 4.1]{llr} it was wondered whether $X\pten Y$ has the SD2P whenever $X$ has the SD2P and $Y$ is non-zero. Notice that a negative answer was given in \cite[Corollary 3.9]{llr2} where it was proved that $L_\infty([0,1])\pten \ell_p^n$ for $2<p<\infty$ and $n\geq 3$ yields a counterexample. However, Theorem \ref{teovectorvalSD2P} gives an affirmative answer in the case that $X$ is a Lipschitz-free space.
 \end{remark}
 
\begin{remark} Let $M$ be a metric space and $X$ be a Banach space.  Since the SD2P and the Daugavet property are equivalent in the context of Lipschitz-free spaces \cite[Theorem 1.5]{am}, a natural question is, in view of Corollary \ref{corocarasd2pvectorval}, whether $\Free(M,X)$ has the Daugavet property whenever $\Free(M)$ has the Daugavet property. Note that, under the assumption that the pair $(M,X^*)$ has the CEP, the answer is affirmative \cite[Corollary 3.12]{gprdaugavet}. In connection with that, and taking into account that $\Free(M,X)=\Free(M)\pten X$, it should be pointed out that the Daugavet property is not preserved, in general, from just one of the factors by taking projective tensor product (c.f. e.g. \cite[Corollary 4.3]{kkw} or \cite[Remark 3.13]{llr2}). We do not know either whether $\mathcal F(M,X)$ has the Daugavet property if both $\mathcal F(M)$ and $X$ have the Daugavet property. For general Banach spaces, it is an open problem whether $X\pten Y$ has the Daugavet property if $X$ and $Y$ have the Daugavet property (c.f. \cite{rtv,werner}).
\end{remark}

\section*{Acknowledgment}

The authors are deeply grateful to the anonymous referee for many helpful suggestions that improved significantly the readability of this paper. In particular, we appreciate the contribution of the referee by providing the shortened proof of Proposition~\ref{prop:octaLipc0}.

\end{document}